\documentclass{amsart}

\usepackage{latexsym,enumerate}
\usepackage{amsmath,amsthm,amsopn,amstext,amscd,amsfonts,amssymb}
\usepackage[ansinew]{inputenc}
\usepackage{verbatim}
\usepackage{graphicx}
\usepackage{pstricks}
\usepackage{wasysym}
\usepackage[all]{xy}
\usepackage{epsfig} 
\usepackage{graphicx,psfrag} 
\usepackage{subfigure}
\usepackage{pstricks,pst-node}

\usepackage{multicol}
\usepackage{color}
\usepackage{colortbl}

\newcommand{\NN}{\mathbb{N}}

\newtheorem{theorem}{Theorem}[section]
\newtheorem{lemma}[theorem]{Lemma}
\newtheorem{proposition}[theorem]{Proposition}
\newtheorem{corollary}[theorem]{Corollary}
\newtheorem{definition}[theorem]{Definition}
\newtheorem{example}[theorem]{Example}
\newtheorem{remark}[theorem]{Remark}

\newcommand{\spb}[1]{\smallskip}
\newcommand{\mpb}[1]{\medskip}
\newcommand{\bpb}[1]{\bigskip}

\renewcommand{\d}{\delta}

\begin{document}
\DeclareGraphicsExtensions{.jpg,.pdf,.mps,.png}

\title{Finite metric and $k$-metric bases on ultrametric spaces}

\author[Samuel G. Corregidor]{Samuel G. Corregidor}
\address{ Facultad CC. Matem\'aticas, Universidad Complutense de Madrid,
Plaza de Ciencias, 3. 28040 Madrid, Spain}
\email{samuguti@ucm.es}

\author[\'{A}lvaro Mart\'{\i}nez-P\'erez]{\'{A}lvaro Mart\'{\i}nez-P\'erez$^{(1)}$}
\address{ Facultad CC. Sociales, Universidad de Castilla-La Mancha,
Avda. Real F\'abrica de Seda, s/n. 45600 Talavera de la Reina, Toledo, Spain}
\email{alvaro.martinezperez@uclm.es}
\thanks{$^{(1)}$ Supported in part by a grant from
Ministerio de Ciencia, Innovaci\'on y Universidades (PGC2018-098321-B-I00), Spain.
}

\date{\today}


\begin{abstract}Given a metric space $(X,d)$, a set $S\subseteq X$ is called a $k$-\emph{metric generator} for $X$ if any pair of different points of $X$ is distinguished by at least $k$ elements of $S$.  A $k$-\emph{metric basis} is a $k$-metric generator of the minimum cardinality in $X$. We prove that ultrametric spaces do not have finite $k$-metric bases for $k>2$. We also characterize when the metric and 2-metric bases of an ultrametric space are finite and, when they are finite, we characterize them. Finally, we prove that an ultrametric space can be easily recovered knowing only the metric basis and the coordinates of the points in it. 
\end{abstract}

\maketitle{}



\section{Introduction}

The idea of metric dimension of a general metric space was introduced in 1953 by Blumental in \cite{B}. The intention is to have a set of fixed points so that for every point in the space, the distances to these fixed points work as a coordinate system, distinguishing every point in the space. The problem is very natural in graphs. Suppose there is a graph and the need of locating something situated in some vertex. A natural solution is to place detectors in certain vertices so that every vertex in the graph is uniquely determined by the distance to these detectors. This idea was developed independently by J. P. Slater in \cite{S1,S2}, where the sets of vertices which uniquely locate every vertex in the graph are called \emph{locating sets}, and Harary and Melter in \cite{HM}, where these sets are called \emph{resolving sets}. Harary and Melter also introduced the name of metric dimension for the cardinality of a minimum resolving set.
This concept has been extensively studied. See, for example, 
\cite{ALA,BC,BR,CH,CGH,CEJO,CSZ,FGO,T,YKR,YR}.

A natural extension of metric dimension is $k$-metric dimension which appears in \cite{ERY}. See also \cite{CM,EYR1,EYR2,EYR3,YER}. The idea now is to strengthen the security of the location asking that we have a family of points such that for every pair of points in the space there are at least $k$ points in our family which are able to distinguish them. 

Let us recall the definitions as they are established now in the literature. A pair of points, $x,y$ in a metric space $(X,d)$ are \emph{distinguished} by another point $z\in X$ if $d(x,z)\neq d(y,z)$. Given a metric space $(X,d)$, a set $S\subseteq X$ is called a $k$-\emph{metric generator} for $X$ if any pair of different points of $X$ is distinguished by at least $k$ elements of $S$.  A $k$-\emph{metric basis} is a $k$-metric generator of minimum cardinality in $X$. A metric generator (resp. metric basis) is just a 1-metric generator (resp. 1-metric basis). The $k$-\emph{metric dimension} is the cardinal of a $k$-metric basis.

A metric space $(X,d)$ is ultrametric if $d(x,y)\leq \max \{d(x,z),d(z,y)\}$
for all $x,y,z\in X$. A classical example of ultrametric space is the space of $p$-adic numbers. There is a well known relation between  trees and ultrametric spaces. One of the basic properties of these spaces is that if two balls intersect then one always contains the other. Therefore, shrinking the diameter of the balls produces a ramification process which can be interpreted as the branching of a tree. In fact, B. Hughes proved that there is a categorical equivalence between rooted, geodesically complete $\mathbb{R}$-trees and bounded, complete ultrametric spaces, \cite{H}, where ultrametric spaces arise as the boundary at infinity of the trees. See also \cite{MM}. This approach is very useful for the representation of ultrametric spaces.

Ultrametric spaces have applications in several branches of science, for example, in physics \cite{AS,MPV}, taxonomy \cite{JS,SS}, psychology \cite{K1,K2,L,Ma}, computer science \cite{M} and phylogenesis \cite{GD}. In fact, any hierarchical system can be modeled by an ultrametric space. Thus, for example, the problem of obtaining a phylogenetic tree or the problem of obtaining a hierarchical clustering for a data set is equivalent to giving an appropriate ultrametric to a set.

In this paper we study the existence of finite $k$-metric bases in ultrametric spaces. We prove that ultrametric spaces have no finite $k$-metric basis for $k>2$. See Theorem \ref{t:kmetric3}. We find that in order to have a finite $k$-metric basis the ultrametric space must satisfy certain condition which we name being pseudo-complete. See Proposition \ref{p:pseudo_infinite2}. In certain natural cases (for example, if the ultrametric space arises as the end space of a geodesically complete simplicial tree), being complete implies being pseudo-complete. However, being pseudo-complete is a necessary but not sufficient condition. As we show in Theorem \ref{t:unpartnered_metricbasis} an ultrametric space has a finite metric basis if and only if it is pseudo-complete and two families of points, the partnered and the unpartnered, are finite. Moreover, there is a finite 2-metric basis if and only if the space is pseudo-complete, the set of partnered points is finite and the set of unpartnered points is empty. Finally,  we characterize the finite metric and 2-metric bases of an ultrametric space in the cases where such a basis exists. See theorems \ref{t:no_unpartnered_2metricbasis2}, \ref{t:no_unpartnered_metricbasis} and \ref{t:unpartnered_metricbasis}.

We also include lower and upper bounds for the $k$-metric dimension for finite ultrametric spaces and show, with examples, that these bounds are sharp.

With a more applicable approach, we obtain that given a pseudo-complete ultrametric space with no unpartnered points and with metric basis $S$, there is a unique minimal subset having $S$ as metric basis and characterize it. See Theorem \ref{t:minimalspace}. Also, we prove that ultrametric spaces can be easily recovered through the metric coordinates, i.e., if we have a metric basis $S$ of an unknown ultrametric space $(X,d)$ and we know the metric coordinates of every point, this is, the distances from every point in $X$ to every point in $S$, then $(X,d)$ is unique and easily rebuilt. See Theorem \ref{t:finite_ultrametricbasis}.

\section{$k$-metric dimension in ultrametric spaces}


Let us denote $dim_k(X)$ the cardinality of any $k$-metric basis. A metric space $(X,d)$ is said to be \emph{$k$-metric dimensional} if $k$ is the largest integer such that there exists a $k$-metric basis for $X$. We denote this as $Dim(X) = k$.

Given an ultrametric space $(X,d)$, let us establish some notations. For every point $x \in X$ and $r>0$ we denote the balls 
$$B_r(x) = \{ p\in X \, | \, d(x,p) < r \} \mbox{ and }  \overline{B}_r(x) = \{ p\in X \, | \, d(x,p) \leqslant r \} .$$

Let us recall some basic properties of ultrametric spaces. See, for example, \cite{H}. 

\begin{proposition}\label{p:elementals}Given an ultrametric space $(X,d)$, then the following properties hold.
	\begin{enumerate}
		\item Given two balls $B_1, B_2 \subseteq X$ such that $B_1 \cap B_2 \neq\emptyset$, then $B_1 \subseteq B_2$ or $B_2 \subseteq B_1$. 
		\item Given a ball $B$, then every point $x\in B$ is the center of $B$.
		\item Every triangle in $X$ is isosceles with a short base.
	\end{enumerate}
\end{proposition}

\begin{lemma}\label{l:distinguished}Given an ultrametric space $(X,d)$, if $p\in X$ and $R>0$, then there is no pair of points in $B_R (p)$ distinguished by points of $X \smallsetminus B_R (p)$.
\end{lemma}

\begin{proof}Given two points $x,y \in B_R (p)$, then for all $z\in X \smallsetminus B_R(p) ,\, d(x,z) , d(y,z) \geqslant R > d(x,y) $ and since $d$ is ultrametric $ d(x,z) = d(y,z)$.
\end{proof}

\begin{definition}\label{d:partner}Given an ultrametric space $(X,d)$, we say that two different points $x,y\in X$ are partners if $d(x,z) \geqslant d(x,y)$ and $d(y,z) \geqslant d(x,y)$ for all $z\in X\smallsetminus\{x,y\}$. A point $x\in X$ is said to be partnered if there exists some other point $y\in X$ such that $x,y$ are partners. A point $x\in X$ is said to be pseudopartnered if it is not partnered but there exists some other point $y\in X$ such that $d(x,z) \geqslant d(x,y)$ for all $z\in X\smallsetminus\{x,y\}$. A point $x\in X$ is unpartnered if $x$ is not partnered nor pseudopartnered. 
\end{definition}

\begin{remark} If $x$ is an adherent point of $X$, then $x$ is unpartnered. However, there exist unpartnered points which are not adherent points as it is shown in Example \ref{ex:ad} below. In a finite ultrametric space there are not unpartnered points.
\end{remark}

\begin{example}\label{ex:ad} Let $X=\NN\cup \{x\}$ with an ultrametric $d$ such that $d(n,m)=1+\frac{1}{\min\{n,m\}}$ for every $n,m\in \NN$ and $d(n,x)=1+\frac{1}{n}$. It is immediate to check that $(X,d)$ is ultrametric. Then, $x$ is an unpartnered point but it is not in the adherence of $X$, which is the empty set.
\end{example}

From now on, we will suppose that every ultrametric space has at least two points. The following propositions show that the existence of a partnered or an unpartnered point has immediate consequences on the existence of finite $k$-metric basis.

\begin{proposition}\label{p:unpartnered_metricbasis}Let $(X,d)$ be an ultrametric space and $x\in X$ be an unpartnered point. If $S$ is a metric basis, then $x\in S$ or $S$ is infinite.
\end{proposition}

\begin{proof}Let $x\in X$ be an unpartnered point. Suppose that there exists $S$ a finite metric basis of $(X,d)$ and $x\notin S$. Since $S$ is finite, then there exists $s \in S$ such that $ 0 < d(x,s) = \min\limits _{z\in S} d(x,z) $. Since $ x $ is not partnered nor pseudopartnered, then there exists $z\in X$ such that $ 0 < d(x,z) < d(x,s)$. Then, by Lemma \ref{l:distinguished}, the points $z$ and $x$ are not distinguished by $S\subseteq X \smallsetminus B_{d(x,s)} (x) $, leading to a contradiction.
\end{proof}

\begin{proposition}\label{p:unpartnered_kmetricbasis}Let $(X,d)$ be an ultrametric space and $k>1$. If there is an unpartnered point $x\in X$, then every $k$-metric basis is infinite.
\end{proposition}

\begin{proof}Let $x\in X$ be an unpartnered point. Suppose that there exists $S$ a finite $k$-metric basis of $(X,d)$. Since $S$ is finite, then there exists $s \in S$ such that $ 0 < d(x,s) = \min\limits _{z\in S\smallsetminus\{x\}} d(x,z) $. Since $ x $ is not partnered nor pseudopartnered, then there exists $z\in X$ such that $ 0 < d(x,z) < d(x,s)$. Then, by lemma \ref{l:distinguished}, the points $z$ and $x$ are not distinguished by $k$ points of $ S $, leading to a contradiction.
\end{proof}

\begin{remark}\label{r:partner_12metricbasis}Consider $x,y$ two points that are partners in an ultrametric space $(X,d)$. Then $x,y$ are in every 2-metric basis of $(X,d)$ and either $x$ or $y$ is in every metric basis of $(X,d)$. 
\end{remark}

\begin{proposition}\label{p:partnered_kmetricbasis}Let $(X,d)$ be an ultrametric space and $k>2$. If there exists a partnered point in $X$, then there is no $k$-metric basis in $(X,d)$.
\end{proposition}

\begin{proof}Suppose that there exists a partnered point $x\in X$. Then, there exists $y\in X$ such that $x,y$ are partners. If $z\in X\smallsetminus\{x,y\}$, then $d(x,z) \geqslant d(x,y)$ and $d(y,z) \geqslant d(x,y)$. By definition of ultrametric, $ d(x,z) = d(y,z) $. Therefore, $x,y$ are distinguished only by themselves.
\end{proof}


The effect in the existence of finite $k$-metric bases of  how are pseudopartnered points organized in the ultrametric space is a more delicate problem to study. To deal with this question we define the following sequences.

\begin{definition}\label{d:distinguished_sequence}Given a pseudopartnered point $x$ in an ultrametric space $(X,d)$, let us define the \emph{$x$-pseudopartnering sequence} $(x_n)_{n\in\mathbb{N}} $ and the associated distances $(d_n)_{n\in\mathbb{N}} $ by induction. First, let

$$ \left\lbrace \begin{array}{c}
x_1 = x \\ 
d_1 = \infty
\end{array}  \right.  $$

Now suppose that $x_n$ and $d_n$ are defined. Then we define $x_{n+1}$ and $d_{n+1}$ as follows: 

\begin{itemize}
\item If $ |B_{d_n} (x_n)| = 1 $ or $d_n = 0$,
$$ \left\lbrace \begin{array}{l}
x_{n+1} = x_n \\ 
d_{n+1} = 0
\end{array}  \right. $$

\item If $|B_{d_n} (x_n)| > 1$, let $ I_n = \{ d(x_n , z) \, |\, z\in B_{d_n} (x_n) \smallsetminus \{x_n\} \} $. 

\begin{itemize}
\item If $I_n$ has no minimum, then for all $y\in X$ there exists $z \in X$ such that $0<d(x_n , z) < d(x_n , y)$. Therefore, $x_n$ is unpartnered and we define:
$$ \left\lbrace \begin{array}{l}
x_{n+1} = x_n \\ 
d_{n+1} = 0
\end{array}  \right. $$

\item If $I_n$ has minimum, then choose any point minimizing the distance:
$$ \left\lbrace \begin{array}{l}
x_{n+1} \in  \{ z\in X \, | \, d(x_n , z) = \min I_n \}  \\ 
d_{n+1} = d(x_n , x_{n+1})
\end{array}  \right. $$
\end{itemize}
\end{itemize}
\end{definition}

\begin{remark} Notice that given a pseudopartnered point $x$ in an ultrametric space $X$, the $x$-pseudopartnering sequence $(x_n)_{n\in \NN}$ is not necessarily unique. Also, $(x_n)_{n\in \NN}$ satisfies that $x_{n+1}\neq x_n$ for every $n$ if and only if $x_n$ is pseudopartnered for every $n$. In this case, we say that the pseudopartnering sequence is \emph{infinite}\footnote{Notice that in order to have an infinite pseudopartnering sequence we have to accept the axiom of countable choice.}. If there exists some point $z\in X$ such that $lim_{n\to \infty}x_n=z$ we say that the pseudopartnering sequence is \emph{convergent}. (Finite pseudopartnering sequences are obviously convergent.) The following examples show that an infinite pseudopartnering sequence may be convergent or not. Moreover, the sequence of distances $(d_n)$ in the infinite pseudopartnering sequence may be convergent to 0 or to some positive number.
\end{remark}

\begin{example}\label{e:natural} Let us define an ultrametric in $\mathbb{N}$. $$ \begin{array}{llll}
	d: & \mathbb{N}\times\mathbb{N} & \longrightarrow & \mathbb{R} \\ 
	& (n,m) & \mapsto & d(n,m) = \left\lbrace \begin{array}{cc}
	\dfrac{1}{\min \{ n,m \}} & \text{if } n\neq m \\ 
	0 & \text{if } n = m
	\end{array}  \right. 
	\end{array}  $$ It is easy to check that every point in $(\mathbb{N}, d)$ is pseudopartnered and every pseudopartnering sequence is infinite and non-convergent with $lim_{n\to \infty}d_n=0$.
	
	If we add an extra poing $\infty$ such that $d(n,\infty)=\frac1n$, then every pseudopartnering sequence converges to $\infty$.
\end{example}

\begin{example} Let $(X,d)$ be the ultrametric space defined in Example \ref{ex:ad}. Then, for every $n\in \NN$, $n$ is pseudopartnered and $x$ is unpartnered. The sequence $(n)_{n\in \NN}$ is a 1-pseudopartnering sequence with $d_{n+1}=1+\frac1n$. Thus, $(n)_{n\in \NN}$ is not convergent. In fact, $lim_{n\to \infty}d_n=1$.
\end{example}

\begin{definition} Given a pseudopartnering sequence $(x_n)_{n\in \NN}$, $z$ is a \emph{pseudolimit} of $(x_n)_{n\in \NN}$ if $d(z,x_n)=d(x_n,x_{n+1})$ for every $n\in \NN$. If a  pseudopartnering sequence has a pseudolimit we say that the pseudopartnering sequence is \emph{pseudoconvergent}.
\end{definition}

\begin{remark}\label{r:plimit} Given a pseudopartnering sequence $(x_n)_{n\in \NN}$, by the properties of the ultrametric, the definition of pseudolimit is equivalent to the following: $z$ is a pseudolimit of $(x_n)_{n\in \NN}$ if $z\in \overline{B}_{d_n+1}(x_{n+1})$ for every $n\in \NN$. 
\end{remark}

Let us denote as $pl(x_n)$ the set of pseudolimits of the pseudopartnering sequence $(x_n)$. Then, by Remark \ref{r:plimit}, the folowing proposition is immediate.

\begin{proposition}\label{pl_ball} Given an infinite pseudopartnering sequence $(x_n)_{n\in \NN}$, then $pl(x_n)=\cap_{n\in \NN}\overline{B}_{d_n}(x_n)$.
\end{proposition}


\begin{remark}\label{r:pl_ball} Given an infinite pseudoconvergent pseudopartnering sequence $(x_n)_{n\in \NN}$, if $\lim_{n\to \infty} d_n=0$, then $pl(x_n)=lim_{n\to \infty}x_n$ (which may be the empty set). Moreover, if $(x_n)_{n\in \NN}$ is a convergent pseudopartnering sequence, its limit is the unique pseudolimit. If $\lim_{n\to \infty} d_n=d>0$, then $pl(x_n)=\overline{B}_d(z)$ for any $z\in pl(x_n)$. 
\end{remark}

\begin{definition}\label{d:complete}An ultrametric space $(X,d)$ is \emph{pseudo-complete} if for every pseudopartnered point $x\in X$ two conditions are satisfied:
\begin{itemize}
	\item[i)]  every $x$-pseudopartnering sequence has a pseudolimit,
	\item[ii)] at least one $x$-pseudopartnering sequence has a finite number of pseudolimits.
\end{itemize}
\end{definition}

\begin{remark}\label{r:finite_complete}If an ultrametric space has no pseudopartnered point, then it is pseudo-complete. Also, every finite ultrametric space is pseudo-complete.
\end{remark}

Notice that being complete does not imply being pseudo-complete. 

\begin{example} Let us define an ultrametric in $\mathbb{N}$. $$ \begin{array}{llll}
	d: & \mathbb{N}\times\mathbb{N} & \longrightarrow & \mathbb{R} \\ 
	& (n,m) & \mapsto & d(n,m) = \left\lbrace \begin{array}{cc}
	1+\dfrac{1}{\min \{ n,m \}} & \text{if } n\neq m \\ 
	0 & \text{if } n = m
	\end{array}  \right. 
	\end{array}  $$ Notice that for every $n\in \mathbb{N}$, $B_1(n)=\{n\}$ and the space is complete. However, every pseudopartnering sequence is infinite and has no pseudolimit.
\end{example}

Also, being pseudo-complete does not imply being complete.

\begin{example} Consider the union of two infinite sequences of points, $X=(x_n)_{n\in \NN}\cup (p_n)_{n\in \NN}$. Let $d(x_i,x_j)=d(p_i,p_j)=d(x_i,p_j)=\frac{1}{\min\{i,j\}}$ and $d(x_i,p_i)=\frac{1}{i+1}$. Thus, it is readily seen that $(X,d)$ is an ultrametric space where $x_i,d_i$ are partners for every $i$. Since there are not pseudopartnered points, the space is pseudo-complete. However, the sequence $(x_n)_{n\in \NN}$ is Cauchy but has no limit. Thus, $X$ is not complete. 
\end{example}

\smallskip

However, under certain (natural) conditions being complete implies being pseudo-complete.
By Remark \ref{r:pl_ball} 
if $(X,d)$ is complete, every pseudopartnering sequence $(x_n)$ such that $lim_{n\to \infty}d_n=0$ has a unique pseudolimit. Then, the following proposition is immediate.

\begin{proposition} Given a complete ultrametric space $(X,d)$, if for every pseudopartnering sequence $(x_n)$, $lim_{n\to \infty}d_n=0$, then $(X,d)$ is pseudo-complete.
\end{proposition}

\begin{corollary} Given a complete ultrametric space $(X,d)$ and the set of distances $\mathcal{D}(X,d) := \{ d(x,y) :\; x,y\in X \}$, if every strictly decreasing sequence $(d_n)_{n\in \NN}$ in $\mathcal{D}$  satisfies that $\lim_{n\to \infty}d_n=0$, then $(X,d)$ is pseudo-complete.
\end{corollary}

In particular, this gives the following corollary which can be interesting when ultrametric spaces arise as end spaces of trees. See \cite{H} for the missing definitions.

\begin{corollary} Given a geodesically complete simplicial tree $T$, then its end space with the usual metric at infinity is pseudo-complete.
\end{corollary}

\smallskip

Being pseudo-complete is a key property not only for the existence of finite metric basis. It also implies the existence of either a partnered or an unpartnered point.

\begin{proposition}\label{p:pseudo_infinite2} If an ultrametric space is not pseudo-complete, then there is no finite $k$-metric basis for any $k\geq 1$.
\end{proposition}

\begin{proof} First, notice that any $(k+1)$-metric basis is a $k$-metric generator. Therefore, $dim_k(X)\leq dim_{k+1}(X)$ for every $k\geq 1$ . Thus, it suffices to prove that there is no finite metric basis.

Suppose there is a pseudopartnered point $x$ and an infinite $x$-pseudopartnering sequence $(x_n)$ with no pseudolimit. Let $S=\{s_1,\dots,s_k\}$ be any finite family of points in $X$. Then, since $(x_n)$ has no pseudolimit, for every $1\leq i\leq k$ there is some $n_i\in \NN$ such that $d(s_i,x_{n_i})>d(x_{n_i},x_{n_{i}+1})$. Let $n_0=\max_i\{n_i\}$. Then, $d(s_i,x_{n_0})>d(x_{n_0},x_{n_{0}+1})$ for every $i$ and $d(s_i,x_{n_0})=d(s_i,x_{m})$ for every $m>n_0$. Therefore, $S$ is not a metric basis of $X$.

Suppose there is a pseudopartnered point $x$ such that every $x$-pseudopartnering sequence $(x_n)$ has an infinite family $Y=\{y_i\}_{i\in I}$ of pseudolimits. Let $S=\{s_1,\dots,s_k\}$ be any finite family of points in $X$ and let us suppose that $S$ distinguishes the points in $Y$.

If for some $\d>0$, $\{B_\d(y_i)\}_{i\in I}$ defines an infinite partition of $Y$, then, the result is trivial from Lemma \ref{l:distinguished} and the pigeonhole principle. Then, for every $n\in \NN$ we can choose inductively some $i_n\in I$ such that $B_{1/n}(y_{i_n})$ contains infinite points in $Y$ with $B_{1/{n+1}}(y_{i_{n+1}})\subset B_{1/n}(y_{i_n})$. By Lemma \ref{l:distinguished}, $S\cap B_{1/n}(y_{i_n})\neq \emptyset$ for every $i$. 
Thus, since $S$ is finite, $\cap_{n\in \NN}B_{1/n}(y_{i_n})\cap S \neq \emptyset$. Suppose, 
$s\in \cap_{n\in \NN}B_{1/n}(y_{i_n})\cap S$. Then, since there are points in $Y$ arbitrarily close to $s$, $s$ is unpartnered. Also, since $s$ is arbitrarily close to some pseudolimit of $(x_n)$, then $d(s,x_n)=d(x_n,x_{n+1})$ for every $n$. Thus, $x,s,s,s,\dots$ defines a (finite) convergent pseudopartnering sequence leading to contradiction.
\end{proof}

\begin{proposition}\label{p:partner_exist}Let $x$ be a pseudopartnered point in an ultrametric space $(X,d)$, and $(x_n)_{n\in\mathbb{N}}$ be a pseudoconvergent $x$-pseudopartnering sequence with a finite number of pseudolimits. Then, there is a pseudolimit $ z $ of $ (x_n) $ that is either partnered or unpartnered (not pseudopartnered).
\end{proposition}

\begin{proof} First, suppose $lim_{n\to \infty} d_n=0$. If the sequence $(x_n)$ is finite, then the proof is immediate by construction. Suppose that the sequence is infinite. Then, by Remark \ref{r:pl_ball}, let $z=lim_{n\to \infty} x_n$ be the unique pseudolimit. Thus, given any point $y\in X$ there is $n_0 \in\mathbb{N}$ such that $ 0 < d(z,x_{n_0}) < d(z,y) $ and $z$ is unpartnered.

Now suppose $lim_{n\to \infty} d_n>0$. Consider any $z\in pl(x_n)=\{z_1,\dots,z_k\}$ and let $d=lim_{n\to \infty}d_n$. By Proposition \ref{pl_ball} and Remark \ref{r:pl_ball}, $pl(x_n)=\cap_{n\in \NN}\overline{B}_{d_n}(x_n)=\overline{B}_d(z)=\{z_1,\dots,z_k\}$. If $z$ is not pseudopartnered, we are done. Otherwise, consider any $z$-pseudopartnering sequence $(z_n)$. By construction, $(z_n)\subset \overline{B}_d(z)=\{z_1,\dots,z_k\}$ and therefore, $(z_n)$ is finite. Hence, it is immediate that $lim_{n\to \infty}(z_n)=z_i$ for some $1\leq i\leq k$ and $z_i$ is either partnered or unpartnered.
\end{proof}

Notice that finite ultrametric spaces are pseudo-complete and do not have unpartnered points.

\begin{corollary}\label{c:partner_exist}Let $(X,d)$ be a pseudo-complete ultrametric space. If there are not unpartened points in $(X,d)$, then there exist partner points.
\end{corollary}

\smallskip

Let us denote $ P(X) $ the set of partnered points and $ UP(X) $ the set of unpartnered points in $(X,d)$.

\begin{remark}Given a pseudo-complete ultrametric space $(X,d)$, then $P(X)\cup UP(X) \neq \emptyset$. However, there are ultrametric spaces without partnered nor unpartnered points, see Example \ref{e:natural}.
\end{remark}

\smallskip

Thus, we can conclude the following:

\begin{theorem}\label{t:kmetric3} If $(X,d)$ is an ultrametric space, then there is no finite $k$-metric basis for $k>2$.
\end{theorem}  

\begin{proof} If $(X,d)$ is not pseudo-complete, it follows from  Proposition \ref{p:pseudo_infinite2}. If $(X,d)$ is pseudo-complete, from Proposition \ref{p:partner_exist}, there is either a partnered or an unpartnered point in $(X,d)$. If there is an unpartnered point, the result follows from  Proposition \ref{p:unpartnered_kmetricbasis}. If there is a partnered point, the result follows from Proposition \ref{p:partnered_kmetricbasis}.
\end{proof}

\smallskip

Let us characterize now which are the metric and 2-metric bases of a pseudo-complete ultrametric space and compute the corresponding metric and 2-metric dimension.

\begin{theorem}\label{t:no_unpartnered_2metricbasis2}If $(X,d)$ is a pseudo-complete ultrametric space without unpartnered points, then $P(X)$ is its unique 2-metric basis.
\end{theorem}

\begin{proof} By Remark \ref{r:partner_12metricbasis}, $P(X)$ is contained in every 2-metric basis of $(X,d)$. Let us see that $P(X)$ is a 2-metric generator. Suppose $x,y\in X$. If $x,y$ are partners, it is trivial that there are two points in $P(X)$, $x,y$, distinguishing them. Otherwise, suppose (without loss of generality) that there exists $z\in X$ such that $d(y,z)<d(x,z)=d(y,x)$. If $y\in P(X)$, let $y'$ be a partner of $y$ and then $y,y'\in P(X)$ distinguish $x,y$. If $y\notin P(X)$, since $(X,d)$ is pseudo-complete, there is a pseudoconvergent $y$-pseudopartnering sequence $(y_n)_{n\in \NN}$. 
Since $X$ has not unpartnered points, by Proposition \ref{p:partner_exist}, there is some $y'\in pl(y_n)\cap P(X)$. Then, $d(y',y)\leq d(z,y)<d(z,x)=d(y,x)=d(y',x)$ and $y'$ distinguishes $x$ and $y$. Let $y''\in P(X)$ such that $y'$ and $y''$ are partners. Then, by the properties of the ultrametric, $d(y',y'')\leq d(y',y)<d(y,x)=d(y',x)=d(y'',x)$ and $y''$ also distinguishes $x$ and $y$.  
\end{proof}

\begin{corollary}\label{c:2dim}If $(X,d)$ is a pseudo-complete ultrametric space without unpartnered points, then 
$$Dim(X) = 2 \ \mbox{ and } \ \dim_2(X) = | P(X) |.$$
\end{corollary}

\begin{corollary}\label{c:2dim_finite}If $(X,d)$ is a finite ultrametric space, then $P(X)$ is its unique 2-metric basis. In particular, 
$$Dim(X) = 2 \ \mbox{ and } \ \dim_2(X) = | P(X) |.$$
\end{corollary}

Notice that being partners is an equivalence relation. Let us define
\[\mathfrak{P}(X) = P(X)/\sim\,\]
where $x\sim z$ if $x,z$ are partners. If $x\in X$ is a partnered point, we denote 
$$[x] := \{ z\in X \,|\, x,z \text{ are partners} \} \cup \{x\}.$$

\begin{theorem}\label{t:no_unpartnered_metricbasis}Let $(X,d)$ be a pseudo-complete ultrametric space without unpartnered points and $\mathfrak{P}(X) = \{ [x_i] \, | \, i\in I \}$. $S$ is a metric basis of $(X,d)$ if and only if there is $(\alpha _i)_{i\in I} \in \prod\limits_{i\in I} [x_i] $ such that $ S = \bigcup\limits _{i\in I} \Big([x_i] \smallsetminus \{\alpha _i \}\Big)$.
\end{theorem}

\begin{proof} By Remark \ref{r:partner_12metricbasis}, if $S$ is a metric basis of  $(X,d)$, then for every $x_i\in P(X)$ there is at most one point $\alpha_i\in [x_i]$ such that $\alpha_i\notin S$. Thus, there is $(\alpha _i)_{i\in I} \in \prod\limits_{i\in I} [x_i] $ such that $ \bigcup\limits _{i\in I} [x_i] \smallsetminus \{\alpha _i \} \subset S$. Then, it suffices to check that given any $(\alpha _i)_{i\in I} \in \prod\limits_{i\in I} [x_i] $, $ S=\bigcup\limits _{i\in I} [x_i] \smallsetminus \{\alpha _i \}$ is a metric generator.


Fix $x,y \in X$. Suppose $x\in P(X)$ (resp. $y\in P(X)$).  Now, if $y \in P(X)$ (resp. $x \in P(X)$) and $[x]=[y]$ then either $x\in S$ or $y\in S$ and $S$ distinguishes $x$ and $y$. Otherwise (if $y\in P(X)$ with $[x]\neq [y]$ or $y\notin P(X)$), consider any $x'\in [x]\cap S$. Then, $d(x',x)<d(x',y)$ and $x'$ distinguishes $x$ and $y$. Thus, if $x$ or $y$ are partnered points we are done. Suppose now that $x,y\notin P(X)$. Since there are no unpartnered points in $X$, $x,y$ are pseudopartnered. By Proposition \ref{p:partner_exist}, consider an $x$-pseudopartnering sequence $(x_n)_{n\in\mathbb{N}} $ with $v\in pl(x_n)\cap P(X)$ and a $y$-pseudopartnering sequence $(y_n)_{n\in\mathbb{N}} $ with $w\in pl(y_n)\cap P(X)$. Notice that, since $x,y$ are not partners,  there is a point $z$ such that either $d(z,x)<d(z,y)$ or $d(z,y)<d(y,x)$. In particular, $d(x_2,x)<d(x,y)$ or $d(y_2,y)<d(y,x)$. Let us assume, with no loss of generality, that $d(y_2,y)<d(y,x)$. Then, since $d(w,y_n)=d(y_n,y_{n+1})$ for every $n$, by the properties of the ultrametric, $d(y_2,y)=d(w,y)<d(w,x)=d(y,x)$. If $w\in S$, we are done. Otherwise, let $w'\in [w]\cap S$. Since $d(w,w')\leq d(w,y)$, by the properties of the ultrametric, $d(w',y)<d(w',x)$ and $w'$ distinguishes $x$ and $y$.  
\end{proof}

\begin{corollary}\label{c:1dim}Given an ultrametric complete-disgintuished space $(X,d)$ without unpartnered points and $\mathfrak{P}(X) = \{ [x_i] \, | \, i\in I \}$, then $$\dim _1 (X) = \sum\limits_{i\in I} (|[x_i]| - 1) .$$
\end{corollary}

\begin{corollary}\label{c:1dim_finite} Let $(X,d)$ be a finite ultrametric space and let $\mathfrak{P}(X)=\{[x_i] \, | \, 1\leq i\leq n\}$. Then 
$S$ is a metric basis of $(X,d)$ if and only if there is  $(\alpha_1,\dots,\alpha_n) \in \prod\limits_{i=1}^n [x_i] $  such that $ S = \bigcup\limits _{i=1}^n \Big([x_i] \smallsetminus \{\alpha _i \}\Big)$.
In particular,  
$$\dim _1 (X) = \sum\limits_{i=1}^n (|[x_i]| - 1).$$
\end{corollary}

\begin{theorem}\label{t:unpartnered_metricbasis}An ultrametric space $(X,d)$ has a finite metric basis $S$ if and only if it is pseudo-complete and the set $P(X)\cup UP(X)$ is finite. In this case, if $\mathfrak{P}(X) = \{ [x_i] \, | \, 1\leq i\leq n \}$, there is $(\alpha_1,\dots,\alpha_n) \in \prod\limits_{i=1}^n [x_i] $ such that $ S = \bigcup\limits _{i=1}^n \Big([x_i] \smallsetminus \{\alpha _i \}\Big)\cup UP(X)$.
\end{theorem}

\begin{proof} By Proposition \ref{p:pseudo_infinite2}, being pseudo-complete is a necessary condition. By Proposition \ref{p:unpartnered_metricbasis} and Remark \ref{r:partner_12metricbasis}, if $S$ is a finite metric basis of  $(X,d)$, then $UP(X)\subset S$ and for every $x_i\in P(X)$ there is at most one point $\alpha_i\in [x_i]$ such that $\alpha_i\not \in S$. Thus, there is $(\alpha_1,\dots,\alpha_n) \in \prod\limits_{i=1}^n [x_i] $ such that $  \bigcup\limits _{i=1}^n \Big([x_i] \smallsetminus \{\alpha _i \}\Big)\cup UP(X)\subset S$.

Suppose that $(X,d)$ is pseudo-complete and $P(X)\cup UP(X)$ is finite. Then, it suffices to check that given any $(\alpha_1,\dots,\alpha_n) \in \prod\limits_{i=1}^n [x_i] $, $S=\bigcup\limits _{i=1}^n \Big([x_i] \smallsetminus \{\alpha _i \}\Big)\cup UP(X)$ is a metric generator.
Consider $x,y \in X$. 

If $x\in UP(X)$, (resp. if $y\in UP(X)$), then $x$ (resp. $y$) distinguishes $x,y$ and the proof is complete. 

If $x\in P(X)$ (resp. $y\in P(X)$) the proof is the same as in Theorem \ref{t:no_unpartnered_metricbasis}.  If $y \in P(X)$ (resp. $x \in P(X)$) and $[x]=[y]$ then either $x\in S$ or $y\in S$ and $S$ distinguishes $x$ and $y$. Otherwise (if $y\in P(X)$ with $[x]\neq [y]$ or $y\notin P(X)$), consider any $x'\in [x]\cap S$. Then, $d(x',x)<d(x',y)$ and $x'$ distinguishes $x$ and $y$. This completes the case where either $x$ or $y$ are partnered points.

Suppose $x,y\notin P(X)\cup UP(X)$ and, by Proposition \ref{p:partner_exist} consider an $x$-pseudopartnering sequence $(x_n)_{n\in\mathbb{N}} $ with $v\in pl(x_n)\cap \Big(P(X)\cup UP(X)\Big)$ and a $y$-pseudopartnering sequence $(y_n)_{n\in\mathbb{N}} $ with $w\in pl(y_n) \cap \Big(P(X)\cup UP(X)\Big)$. 
Notice that, since $x,y$ are not partners,  there is a point $z$ such that either $d(z,x)<d(z,y)$ or $d(z,y)<d(y,x)$. In particular, $d(x_2,x)<d(x,y)$ or $d(y_2,y)<d(y,x)$. Let us assume, with no loss of generality, that $d(y_2,y)<d(y,x)$. Then, since $d(w,y_n)=d(y_n,y_{n+1})$ for every $n$, by the properties of the ultrametric, $d(y_2,y)=d(w,y)<d(w,x)=d(y,x)$. If $w\in S$, we are done. Otherwise, $w\in P(X)$ and let $w'\in [w]\cap S$. Since $d(w,w')\leq d(w,y)$, by the properties of the ultrametric, $d(w',y)<d(w',x)$ and $w'$ distinguishes $x$ and $y$.  
\end{proof}

\smallskip

Let us see now the bounds for the metric and 2-metric dimension of a finite metric space. The following proposition is pretty much immediate: 

\begin{proposition}\label{p:bounds} Given a finite ultrametric space $(X,d)$ with $n$ points, then
\begin{equation}\label{eq1}
1\leq dim_1(X) \leq n-1,
\end{equation}
\begin{equation}\label{eq2}
2\leq dim_2(X) \leq n
\end{equation}
\end{proposition}

\begin{proof} The lower bound in \eqref{eq1} is trivial. The upper bound of in \eqref{eq1} follows immediately by Corollary \ref{c:partner_exist} and Theorem \ref{t:no_unpartnered_metricbasis}.

The bounds in \eqref{eq2} are trivial.
\end{proof}

The bounds from Proposition \ref{p:bounds} are sharp as it is shown in Example \ref{ex:sharp} below.

\smallskip

\begin{example}\label{ex:sharp} Given $n\in\mathbb{N}$ consider the space $X=\{1,2,\dots,n\}$ and the ultrametric $d_1(a,b)=1$ for every $a,b\in X$. 
It is immediate to check that, $\dim _1 (X,d_1) = n-1$ and $\dim _2 (X,d_1) = n$.

\smallskip 

Consider now $X$ with the ultrametric $d_2(a,b)=\frac{1}{\min\{a,b\}}$  for every $a,b\in X$. 
It is readily seen that $\dim _1 (X_2) = 1$ and $\dim _2 (X_2) = 2$.
\end{example}

\smallskip

We conclude the paper with two results which might be interesting por applications of metric basis of ultrametric spaces. The first one, considers the problem of neglecting points in a pseudo-complete ultrametric space without changing the metric basis. The second one considers the problem of the reconstruction of the whole space through the metric basis. An interesting property of ultrametric spaces is that given a metric basis of an ultrametric space and the coordinates of the points in this metric basis, then the original space can be easily recovered.

\begin{theorem}\label{t:minimalspace}Given a pseudo-complete ultrametric space $(X,d)$
without unpartnered points and metric basis $S$, 
then there exists a unique minimal ultrametric subspace $(X_m,d|_{X_m})$ such that $S$ is a metric basis of $(X_m,d|_{X_m})$. In fact, $X_m = P(X)$.
\end{theorem}

\begin{proof}Let $S$ be a metric basis of a finite ultrametric space $(X ,d)$ and $\mathfrak{P}(X) = \{ [x_i] \, | \, i\in I \}$ the classes of partnered points. By Theorem \ref{t:no_unpartnered_metricbasis} there exists $(\alpha _i)_{i\in I} \in \prod\limits_{i\in I} [x_i] $ such that $ S = \bigcup\limits _{i\in I} \Big([x_i] \smallsetminus \{\alpha _i \}\Big)$. Consider $X_m=\bigcup\limits _{i\in I} [x_i]  = P(X) $ and $ d_m = d|_{X_m\times X_m}$. Obviously $P(X_m)=P(X)$ and $S$ is a metric basis of $(X_m , d_m)$. 

Consider any ultrametric subspace $ (X_p , d_p ) $ of $ (X , d)$ such that $S$ is a metric basis of $ (X_p , d_p ) $ and suppose $X_m\not \subset X_p$.  Given $x\in X_m \smallsetminus X_p$, since $S$ is a metric basis of $ (X_p , d_p ) $, then $x\notin S$. Therefore, $ x = \alpha_{i_0} $ for some 
$ i_0 \in I$. Now, if $|[\alpha_{i_0}]|=2$, let $y$ be the unique partner of $\alpha_{i_0}$ in $X$. Then, $y\in S$ but $y$ is not partnered in $X_p$ and therefore, by Theorem \ref{t:no_unpartnered_metricbasis}, $S$ is not a metric basis of $X_p$ leading to contradiction. If $|[\alpha_{i_0}]|>2$, since $ [\alpha_{i_0}] \smallsetminus \{ \alpha_{i_0}\} \subset S \subset X_p$, $ [\alpha_{i_0}] \smallsetminus \{ \alpha_{i_0}\}\in \mathfrak{P}(X_p) $. Thus, by Theorem \ref{t:no_unpartnered_metricbasis}, $S$ is not a metric basis of $(X_p, d_p)$ leading to contradiction.
\end{proof}

\begin{corollary} Given a pseudo-complete ultrametric space $(X,d)$ without unpartnered points and metric basis $S$ and a subspace $X'\subset X$, then $S$ is a metric basis of $X'$ if and only if $P(X)\subset X'$.
\end{corollary}

In particular, these two results work for finite ultrametric spaces:

\begin{corollary}\label{c:finite_minimalspace}Given a finite ultrametric space $(X,d)$ with metric basis $S$, then there exists a unique minimal ultrametric subspace $(X_m,d|_{X_m})$ such that $S$ is a metric basis of $(X_m,d|_{X_m})$. In fact, $X_m = P(X)$.
\end{corollary}



\begin{corollary} Given a finite ultrametric space $(X,d)$ with metric basis $S$ and a subspace $X'\subset X$, then $S$ is a metric basis of $X'$ if and only if $P(X)\subset X'$.
\end{corollary}

\begin{theorem}\label{t:finite_ultrametricbasis}Given a metric basis $S$ of an ultrametric space $(X,d)$. If $d'$ is any other ultrametric of $X$ such that $d'(x,s) = d(x,s)$ for all $x\in X,\, s\in S$, then $d=d'$. Moreover, for any $x,y\in X$ and any $s\in S$ that distinguishes $x,y$, $d(x,y)=\max\{d(x,s),d(y,s)\}$.
\end{theorem}

\begin{proof}Fix $x,y\in X$ and consider $s\in S$ such that $x,y$ are distinguished by $s$. Suppose, without loss of generality, that $ d' (s,x) = d (s,x) < d (s,y) = d' (s,y) $. Since $d$ is ultrametric, then $d (x,y) = d (s,y) $. By the same argument, $ d' (x,y) = d' (s,y) $. Thus, $ d'(x,y) = d(x,y)$.
\end{proof}

\end{document}